\documentclass[a4paper]{amsart}

\newcommand{\N}{{\mathbb N}}
\newcommand{\Z}{{\mathbb Z}}

\usepackage{amssymb}

\usepackage[all]{xy}
\SelectTips{cm}{}

\theoremstyle{plain}
\newtheorem{thm}{Theorem}[section]

\newtheorem{lem}[thm]{Lemma}
\newtheorem{prop}[thm]{Proposition}

\theoremstyle{definition}
\newtheorem{dfn}[thm]{Definition}

\theoremstyle{remark}
\newtheorem{exa}[thm]{Example}
\newtheorem{rem}[thm]{Remark}

\DeclareMathOperator{\clr}{clr}
\DeclareMathOperator{\sgn}{sgn}
\DeclareMathOperator{\id}{id}
\DeclareMathOperator{\out}{out} \DeclareMathOperator{\inn}{in}

\newcommand{\rpd}{\Omega_p}

\newcommand{\To}{\longrightarrow}
\newcommand{\A}{\mathcal{A}b}
\newcommand{\Ch}{\mathcal{C}h}
\newcommand{\Oper}{\mathcal{O}per}
\newcommand{\dCh}{d\mathcal{C}h}
\newcommand{\dAb}{d\mathcal{A}b}
\newcommand{\X}{\Theta}
\numberwithin{equation}{section}
\numberwithin{figure}{section}

\begin{document}

\title[Dold-Kan correspondence for dendroidal abelian groups]{Dold-Kan correspondence for \\ dendroidal abelian groups}

\author[J. J. Guti\'errez]{Javier J. Guti\'errez}
\address{Centre de Recerca Matem\`atica, Apartat 50, 08193 Bellaterra (Barcelona), Spain}
\email{JGutierrez@crm.cat}
\author[A. Lukacs]{Andor Lukacs}
\author[I. Weiss]{Ittay Weiss}
\address{Mathematisch Instituut, Postbus 80.010, 3508 TA Utrecht, The Netherlands}
\email{A.Lukacs@math.uu.nl}
\email{I.Weiss@uu.nl}


\begin{abstract}
We prove a Dold-Kan type correspondence between the category of
dendroidal abelian groups and a suitably constructed category of
dendroidal complexes. Our result naturally extends the classical
Dold-Kan correspondence between the category of simplicial abelian
groups and the category of non-negatively graded chain complexes.
\end{abstract}

\maketitle

\section{Introduction}
The classical Dold-Kan correspondence \cite{dold, kan} states that there is an
equivalence between the category $s\A$ of simplicial abelian groups
and the category $\Ch$ of non-negatively graded chain complexes of
abelian groups. This equivalence is given by the normalized functor
$N_s$ that sends a simplicial abelian group $A$ to the chain complex
$$
N_s(A)(n)=\bigcap_{i=0}^{n-1}\ker(\partial_i^*),
$$
where $\partial_i^*\colon A_n\longrightarrow A_{n-1}$ are the induced face maps, and the differential is defined by $d=(-1)^n\partial_n^*$. Moreover, the
functor $N_s$ has a right adjoint $\Gamma_s$ such that the adjoint pair
$$
\xymatrix{
N_s\colon s\A \ar@<3pt>[r] & \ar@<3pt>[l] \Ch\colon \Gamma_s
}
$$
is an adjoint equivalence. There exist several generalizations of
this correspondence which characterize simplicial objects in more
structured algebraic categories in terms of the appropriate
algebraic chain objects \cite{pirashvili, slominska}.

In this paper, we extend this result to the framework of (planar)
dendroidal sets. The category of dendroidal sets is a presheaf
category on a certain category of trees~$\Omega_p$. They generalize
simplicial sets in a suitable way for studying the homotopy theory
of coloured operads and their algebras \cite{moerdijkweiss}. The
idea behind the notion of dendroidal sets is that in the same way as
simplicial sets help us understanding categories via the nerve
functor, there should be an analogous notion for studying coloured
operads as generalization of categories. Indeed much of the
fundamentals of simplicial sets that relate to category theory
extend to dendroidal sets. In \cite{moerdijkweiss2}, Moerdijk and
the third-named author develop the theory of inner Kan complexes in
the category of dendroidal sets. Inner Kan complexes in the category
of simplicial sets were first introduced by Boardman and Vogt in
\cite{BV73}. In a later paper \cite{joyal} Joyal began a
reexamination of inner Kan complexes under the name quasi
categories. One of the results of his research is the establishment
of a Quillen model category structure on simplicial sets in which
the fibrant objects are precisely the inner Kan complexes. The same
model structure was also obtained by Lurie in his work on higher
topos theory \cite{lurie}. In \cite{moerdijkcisinski} Cisinski and
Moerdijk establish a Quillen model structure on the category of
dendroidal sets in which the fibrant objects are precisely the inner
Kan complexes. The results in this paper add to the above mentioned
theory by showing that the Dold-Kan correspondence similarly extends
to dendroidal sets.

There is a fully faithful
embedding $i\colon\Delta\longrightarrow\Omega_p$, where $\Delta$
denotes the simplicial category, inducing an adjoint pair
$$
\xymatrix{
i_!\colon s\A \ar@<3pt>[r] & \ar@<3pt>[l] d\A\colon i^*,
}
$$
where $d\A$ denotes the category of (planar) dendroidal abelian
groups.

The category of chain complexes is not big enough to define a Dold-Kan correspondence for the category of dendroidal abelian groups. In order to
solve this problem, we introduce the category of dendroidal complexes $d\Ch$. The category of dendroidal complexes is a category of $\Omega_p$-graded abelian groups together with structure maps induced by the face maps in $\Omega_p$ and satisfying certain conditions. In the same way as the category of dendroidal abelian groups extends the category of simplicial abelian groups, the category of dendroidal complexes extends the category of chain complexes, i.e., there are adjoint functors
$$
\xymatrix{
j_!\colon \Ch \ar@<3pt>[r] & \ar@<3pt>[l] d\Ch\colon j^*,
}
$$
where $j^*$ is the restriction functor and its left adjoint $j_!$ is `extension by zero'.

We define a normalized functor $N\colon d\A\longrightarrow d\Ch$ and a right adjoint $\Gamma$ and prove that they form an adjoint equivalence of
categories. We also show that there is a commutative diagram of adjoint functors
$$
\xymatrix@R+3mm@C+3mm{ s\A\ar@<2pt>[r]^-{i_!}\ar@<-2pt>[d]_{N_s} &\dAb\ar@<-2pt>[d]_N\ar@<2pt>[l]^-{i^*} \\
\Ch\ar@<2pt>[r]^-{j_!}\ar@<-2pt>[u]_{\Gamma_s}&\dCh\ar@<2pt>[l]^-{j^*}\ar@<-2pt>[u]_{\Gamma},}
$$
relating the classical Dold-Kan correspondence with the dendroidal one.

\bigskip
\noindent\textbf{Acknowledgements.} We would like to thank Ieke
Moerdijk for helpful suggestions and comments. The first-named
author wishes to thank the Department of Mathematics at Utrecht
University for its hospitality.
\section{A formalism of trees}

A \emph{tree} is a connected finite graph with no loops. A vertex in
a graph is called \emph{unary} if it has only one edge attached to
it. All the trees we will consider are \emph{rooted trees}, i.e.,
equipped with a distinguished unary vertex called the \emph{output}
and a set of unary vertices (not containing the output vertex)
called the set of \emph{inputs}.

When drawing trees, we will represent them as oriented towards the output, drawn at the bottom, and we will delete the
output and input vertices from the picture. From now on, the term `vertex' in a tree will always refer to a
remaining vertex. Given a tree $T$, we denote by $V(T)$ the set of vertices of $T$ and by $E(T)$ the set of edges of $T$.

The edges attached to the deleted input vertices are called \emph{input edges}
or \emph{leaves}; the edge attached to the deleted output vertex is called \emph{output edge} or \emph{root}. The rest of
the edges are called \emph{inner edges}. The root induces an obvious direction in the tree, `from the leaves towards the
root'. If $v$ is a vertex of a finite rooted tree, we denote by $\out(v)$ the unique
outgoing edge and by $\inn(v)$ the set of incoming edges (note that $\inn(v)$ can be empty).
The cardinality of $\inn(v)$ is called the \emph{valence of $v$}, the element of $\out(v)$ is the
\emph{output of $v$} and the elements of $\inn(v)$ are the \emph{inputs of~$v$}.

As an example, consider the following picture of a tree $T$:

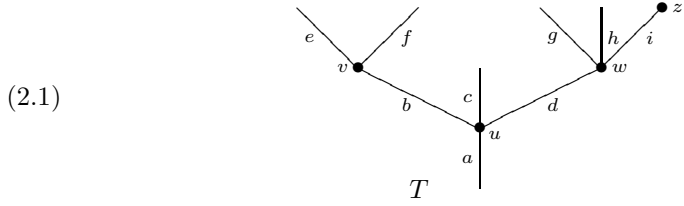
\begin{equation}
\xy<0.08cm, 0cm>: (0,-15)*{}="1"; (0,-5)*=0{\bullet}="2";
(-20,5)*=0{\bullet}="3"; (0,5)*=0{}="4"; (20,5)*=0{\bullet}="5";
(-30,15)*{}="6"; (-10,15)*{}="7"; (10,15)*{}="8"; (20,15)*{}="9";
(30,15)*=0{\bullet}="10"; (-2,-10)*=0{\scriptstyle a};
(2.5,-6)*=0{\scriptstyle u}; (-2,0)*=0{\scriptstyle c};
(-22.5,5)*=0{\scriptstyle v}; (23,5)*=0{\scriptstyle w};
(-28,10)*=0{\scriptstyle e}; (28,10)*=0{\scriptstyle i};
(-12,10)*=0{\scriptstyle f}; (12,-1)*=0{\scriptstyle d};
(-12,-1)*=0{\scriptstyle b}; (12,10)*=0{\scriptstyle g};
(22,10)*=0{\scriptstyle h}; (32.5,15)*=0{\scriptstyle z};
(-10,-15)*{T}; "1";"2" **\dir{-}; "2";"3" **\dir{-}; "2";"4"
**\dir{-}; "2";"5" **\dir{-}; "3";"6" **\dir{-}; "3";"7" **\dir{-};
"5";"8" **\dir{-}; "5";"9" **\dir{-}; "5";"10" **\dir{-};
\endxy
\label{treeexample}
\end{equation}
The output vertex at the edge $a$ and the input vertices at $e$, $f$, $c$, $g$ and $h$ have been deleted.
This tree has four vertices $u$, $v$, $w$ and $z$ of respective valences 3, 2, 3 and 0. It also has five input
edges or leaves, namely $e$, $f$, $c$, $g$ and $h$. The edges $b$ and $d$ are inner edges and the edge $a$ is the root.

Since every planar representation of a rooted tree
comes naturally with an ordering of the inputs of any given vertex
(from left to right), we can give the following definition:

\begin{dfn}
A \emph{planar rooted tree} is a rooted tree $T$ together with a linear ordering of
$\inn(v)$ for each vertex $v$ of $T$.
\end{dfn}

In the rest of the paper we will work with \emph{planar rooted trees}, unless otherwise stated, and for the sake of
bookkeeping whenever it is obvious from the context we will refer to them as trees.

\section{Coloured operads and the dendroidal category}
The dendroidal category $\Omega$ was introduced in
\cite{moerdijkweiss,Wei07} as an extension of the simplicial
category $\Delta$. The category $\Omega$ is a category of trees. Its
objects are (non-planar) rooted trees and the set of morphisms
between two trees is given by the set of maps between the symmetric
coloured operads associated to each of them. The presheaves in
$\Omega$, called \emph{dendroidal sets}, are very useful in the
study of operads and their algebras in the framework of homotopy
theory. Since the terminology on dendroidal sets is  recent, we will
recall all the needed parts of it here.

In this section, we are going to describe a a variation on the
category $\Omega$, called the \emph{planar dendroidal category},
which we denote by $\rpd$, whose objects are the planar rooted
trees. More concretely, let $P\colon \Omega^{\rm op}\longrightarrow
\mathcal{S}ets$ be the presheaf on $\Omega$ that sends each tree to
its set of planar structures. Then $P(T)$ is a torsor under ${\rm
Aut}(T)$ for every tree $T$, where ${\rm Aut}(T)$ denotes the set of
automorphisms of $T$, and $\rpd$ is equal to the category of
elements $\Omega/P$, whose objects are pairs $(T, x)$, with $x\in
P(T)$, and a morphism between two objects $(T,x)$ and $(S,y)$ is
given by a morphism $f\colon T\longrightarrow S$ in $\Omega$ such
that $P(f)(y)=x$.

In order to have a better understanding of the morphisms in $\rpd$
we need the notion of coloured operad. For us all coloured operads
come \emph{without} an action of the symmetric group. Usually they
are referred as non-symmetric coloured operads or non-symmetric
multicategories in the literature (see, for example,
\cite{bergermoerdijk, markl}).

\subsection{Coloured operads}
A coloured operad $P$ consists of a set of \emph{colours}, denoted
by $\clr(P)$, together with a set of \emph{operations}
$P(c_1,\ldots,c_n;c)$ for every $n\ge 0$ and each ordered
$(n+1)$-tuple of colours $(c_1,\ldots, c_n; c)$, and a distinguished
operation $\id_c$ in $P(c;c)$ for every color $c$, called the
\emph{identity on c}. These operations are related by means of
\emph{composition product} maps $\circ_i$
$$
\xymatrix{P(c_1,\ldots, c_i,\ldots,c_n;c)\times P(a_1,\ldots,a_m;c_i)\ar[d]_-{\circ_i}\\
P(c_1,\ldots, c_{i-1},a_1,\ldots,a_m,c_{i+1},\ldots,c_n;c)}
$$
for every $c_1, \ldots, c_n, c$ and $a_1,\ldots, a_m$ in $C$, and
every $1\leq i\leq n$. The composition product maps are subject to
the usual associativity and unitary compatibility relations; see,
for example, \cite{leinster}. If $u\in P(c_1,\ldots,c_n;c)$ is an
operation, then the sequence $(c_1,\ldots,c_n)$ is called the
\emph{input} of $u$ and $c$ is called the \emph{output} of $u$.

A \textit{map of coloured operads} $\varphi\colon P\To Q$ consists
of a map $\clr(\varphi)\colon \clr(P)\To \clr(Q)$ between the
colours and maps $P(c_1,\ldots, c_n;c)\To Q(\varphi(c_1),\ldots,
\varphi(c_n);\varphi(c))$ on the operations, such that $\varphi$
sends units to units and is compatible with the composition product
of operations. We denote by $\Oper$ the category of coloured
operads.

Coloured operad generalize small categories. Any small category
$\mathcal{C}$ can be viewed as a coloured operad
$j_!(\mathcal{C})$, where $\clr(j_!(\mathcal{C}))$
is precisely the set of objects of $\mathcal{C}$ and the only
operations are $j_!(\mathcal C)(A;B)=\mathcal{C}(A,B)$. There is a pair of
adjoint functors
$$
\xymatrix{
j_!\colon \mathcal{C}at \ar@<3pt>[r] & \ar@<3pt>[l] \Oper\colon j^*
}
$$
where $\mathcal{C}at$ denotes the category of small categories. The left
adjoint $j_!$ is full and faithful. The right adjoint $j^*$ sends a
coloured operad $P$ to the category $j^*(P)$ whose objects are the
colours of $P$ and whose morphisms are the unary operations of $P$.

In a similar way we can think of operations as generalizations of the
notion of morphisms in categories. If $P$ is a coloured
operad, a suitable intuitive way to depict an operation
$u\in P(c_1,\ldots, c_n;c)$ is to draw the tree
$$
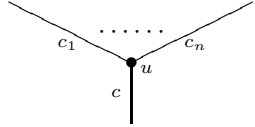
\xy<0.08cm, 0cm>:
(0,0)*{}="1";
(0,10)*=0{\bullet}="2";
(-20,20)*{}="3";
(20,20)*{}="4";
"1";"2" **\dir{-};
"2";"3" **\dir{-};
"2";"4" **\dir{-};
(-2.5,5)*=0{\scriptstyle c};
(-10.5, 13)*=0{\scriptstyle c_1};
(10.5, 13)*=0{\scriptstyle c_n};
(2.5, 9)*=0{\scriptstyle u};
(0, 15)*=0{\cdots\cdots};
\endxy
$$
Note that the case $n=0$ is allowed in the definition of operations,
thus there are operations $u\in P(\ ;c)$ with no input and one
output. Operations of this type can be thought as playing the role
of constants. This is represented by the tree
$$
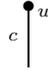
\xy<0.08cm, 0cm>:
(0,0)*{}="1";
(0,10)*=0{\bullet}="2";
"1";"2" **\dir{-};
(-2.5,5)*=0{\scriptstyle c};
(2.5, 9)*=0{\scriptstyle u};
\endxy
$$

To have a good intuition on the nature of the associativity
relations of the composition product, we can use pictures of trees.
A composition product map $\circ_i$ can be thought of as taking  two
operations $u, v$ as its inputs, depicted like a tree with two
vertices
$$
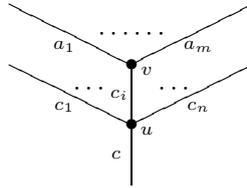
\xy<0.08cm, 0cm>:
(0,0)*{}="1";
(0,10)*=0{\bullet}="2";
(0,20)*=0{\bullet}="2a";
(-20,20)*{}="3";
(20,20)*{}="4";
(-20,30)*{}="5";
(20,30)*{}="6";
"1";"2" **\dir{-};
"2";"3" **\dir{-};
"2";"4" **\dir{-};
"2";"2a" **\dir{-};
"2a";"5" **\dir{-};
"2a";"6" **\dir{-};
(-2.5,5)*=0{\scriptstyle c};
(-11, 13)*=0{\scriptstyle c_1};
(11, 13)*=0{\scriptstyle c_n};
(2.5, 9)*=0{\scriptstyle u};
(-2,15)*=0{\scriptstyle c_i};
(-7, 16)*=0{\cdots};
(7, 16)*=0{\cdots};
(-11, 23)*=0{\scriptstyle a_1};
(11, 23)*=0{\scriptstyle a_m};
(2.5, 19)*=0{\scriptstyle v};
(0, 25)*=0{\cdots\cdots};
\endxy
$$
and producing a new operation by `grafting' the edge $c_i$ (the sole
inner edge of the tree) as follows
$$
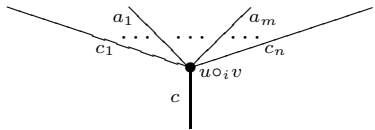
\xy<0.08cm, 0cm>:
(0,0)*{}="1";
(0,10)*=0{\bullet}="2";
(-10,20)*{}="3";
(10,20)*{}="4";
(-30,20)*{}="5";
(30,20)*{}="6";
"1";"2" **\dir{-};
"2";"3" **\dir{-};
"2";"4" **\dir{-};
"2";"5" **\dir{-};
"2";"6" **\dir{-};
(-2.5,5)*=0{\scriptstyle c};
(-14, 13)*=0{\scriptstyle c_1};
(12, 18)*=0{\scriptstyle a_m};
(-11, 18)*=0{\scriptstyle a_1};
(14, 13)*=0{\scriptstyle c_n};
(5, 9)*=0{\scriptstyle u\circ_i v};
(0, 15)*=0{\cdots};
(9, 15)*=0{\cdots};
(-9, 15)*=0{\cdots};
\endxy
$$
The associativity relation states that whenever we have three operations depicted on a tree with three vertices (thus
with two inner edges) then grafting both of the inner edges does not depend on the chosen order.

The definition we give for the composition product of a coloured operad is not the common one in the literature, but it is
equivalent to it (see the same definition in \cite{markl} for the one-colour case, or an alternative definition for the
general case in~\cite{bergermoerdijk,leinster,moerdijkweiss}).

\subsection{The planar dendroidal category }\label{the basic category}
Any tree $T$ gives rise to a (non-symme\-tric) coloured operad which
we denote by $\Omega_p(T)$. The set of colours is the set of edges of
$T$ and the operations of $\Omega_p(T)$ are freely generated by the
vertices of $T$. That is, if $v$ is a vertex with the ordered
sequence of input edges $(c_1, c_2, \ldots, c_n)$ and output edge
$c$ then $\Omega_p(T)(c_1,\ldots,c_n;c)=\{v\}$. All the other
operations are identities or compositions of the previous ones. As a
consequence, any set of operations is either empty or it contains
only one element.

For example, if $T$ is the tree depicted in (\ref{treeexample}), then there are four generating operations:
$\Omega_p(T)(e,f;b)=\{v\}$, $\Omega_p(T)(b,c,d;a)=\{u\}$, $\Omega_p(T)(g,h,i;d)=\{w\}$ and $\Omega_p(T)(\ ;i)=\{z\}$.
For the other sets of operations, we have $\Omega_p(T)(e,f,c,d;a)=\{u\circ_b v\}$ and so on.
Observe that the $b$ in $\circ_b$ refers to the position of $b$ in the ordered sequence of input edges of $u$, i.e.,
$\circ_b=\circ_1$, and in general this notation is not ambiguous since $\Omega_p(T)(c_1,\ldots, c_n;c)=\emptyset$ if $c_i=c_j$
for some $i\neq j$ or $c_i=c$. We will keep using this notation in what follows.

Now, consider the category whose objects are planar rooted trees and whose morphisms $R\To T$ are given by coloured
operad maps $\Omega_p(R)\To\Omega_p(T)$. Note that if $\Omega_p(R)\longrightarrow \Omega_p(T)$ is an isomorphism, then the
non-symmetric operad structures imply that $R$ and $T$ have the same planar shape and they differ only by the names of
their vertices and edges. We define the \textit{planar dendroidal category} as a skeleton of this category.

\begin{dfn}
The \emph{planar dendroidal category} $\rpd$ is the category whose objects are isomorphism classes of
planar rooted trees and the morphisms are given by coloured operad maps, i.e.,
$$
\rpd(R,T)=\Oper(\Omega_p(R), \Omega_p(T))
$$
for every two trees $R$ and $T$ in $\rpd$.
\end{dfn}

In order to simplify the notation, we will omit mentioning
isomorphism classes and we will write $T$ instead of $[T]$.

The simplicial category $\Delta$ can be viewed as the full
subcategory of $\mathcal{C}at$, the category of small categories,
spanned by $\{[n]\,|\, n\ge 0\}$, where $[n]$ is the category whose
objects are $\{0,1,\ldots, n\}$ and for $0\le i,j\le n$ there is
only one arrow $i\longrightarrow j$ if $i\le j$. The category
$\Omega_p$ extends the category $\Delta$. Indeed, if we denote by
$L_n$ the linear tree with $n$ vertices and $n+1$ edges,

then the simplicial category can be identified with the full subcategory of $\rpd$ consisting of linear trees as objects
by means of a functor
\begin{equation}
i\colon \Delta\longrightarrow \Omega_p
\label{embedding of delta}
\end{equation}
sending $[n]$ to $L_n$, which is a full and faithful embedding, since $j_!([n])\cong\Omega_p(L_n)$.

Observe that there is a canonical order on the edges of a linear
tree by numbering them in increasing order from bottom to top.
Whenever we speak of an order on the edges of a linear tree we will
be referring to this order.

The morphisms in $\rpd$ are generated by two types of maps called \emph{faces} and
\emph{degeneracies}, which we discuss in the following sections.

\subsection{Face maps} Suppose that $T$ is (a representative of) an object of $\rpd$ and
$b$ is an inner edge of $T$, as in~(\ref{innerface}) below. Denote by $T/b$ the tree obtained from $T$ by contracting
$b$. There is a face map corresponding to this operation $\partial_b:T/b\To T$ which is the inclusion on the colours of $\Omega_p(T/b)$
and on the generating operations of $\Omega_p(T/b)$, except for the operation  $u$, which is sent to $v\circ_b w$.
The face maps $\partial_b$ associated to inner edges in such a way are the called \emph{inner faces} of~$T$.

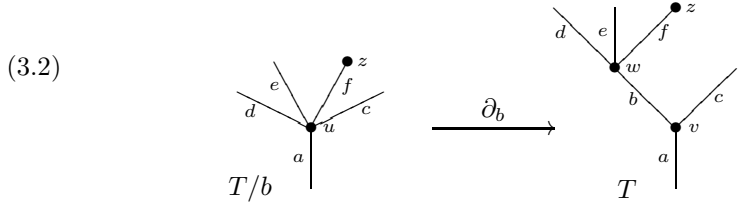
\begin{equation}
\xy<0.08cm,0cm>:
(-12,-4)*=0{}="t1";
(-6,1)*=0{}="t2";
(6,1)*=0{\bullet}="t3";
(12,-4)*=0{}="t4";
(0,-10)*=0{\bullet}="t5";
(0,-20)*=0{}="t6";
(40,10)*=0{}="5";
(50,10)*=0{}="6";
(60,10)*=0{\bullet}="7";
(50,0)*=0{\bullet}="8";
(70,0)*=0{}="9";
(60,-10)*=0{\bullet}="10";
(60,-20)*=0{}="11";
(-10,-7)*{\scriptstyle d};
(-6,-3)*{\scriptstyle e};
(6,-3)*{\scriptstyle f};
(9,-7)*{\scriptstyle c};
(-2,-15)*{\scriptstyle a};
(3,-10)*{\scriptstyle u};
(41,6)*{\scriptstyle d};
(48,6)*{\scriptstyle e};
(58,6)*{\scriptstyle f};
(53,-5)*{\scriptstyle b};
(67,-5)*{\scriptstyle c};
(63,-10)*{\scriptstyle v};
(58,-15)*{\scriptstyle a};
(53,0)*{\scriptstyle w};
(30,-7)*{\partial_b};
(8.5,1)*{\scriptstyle z};
(62.5,10)*{\scriptstyle z};
(-10,-20)*=0{T/b};
(52,-20)*=0{T};
(-10,-25)*=0{};
"t1";"t5" **\dir{-};
"t2";"t5" **\dir{-};
"t3";"t5" **\dir{-};
"t4";"t5" **\dir{-};
"t5";"t6" **\dir{-};
"5";"8" **\dir{-};
"6";"8" **\dir{-};
"7";"8" **\dir{-};
"8";"10" **\dir{-};
"9";"10" **\dir{-};
"10";"11" **\dir{-};
{\ar(20, -10)*{};(40,-10)*{}};
\endxy
\label{innerface}
\end{equation}

Now suppose that $T$ is (a representative of) an object of $\rpd$ and $w$ is a vertex of $T$ with exactly one inner edge
attached to it as in~(\ref{outerface}). It follows that if we remove from $T$ the vertex $w$ and all the outer edges
attached to it, we obtain a new tree $T/w$. There is a face map associated to this operation $\partial_w:T/w\To T$ which is
the inclusion both on the colours and on the generating operations of $\Omega_p(T/w)$. Face maps of this type are called \emph{outer faces} of~$T$.

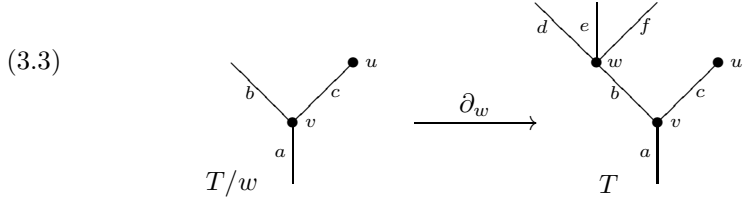
\begin{equation}
\xy<0.08cm,0cm>:
(-10,0)*=0{}="1";
(10,0)*=0{\bullet}="2";
(0,-10)*=0{\bullet}="3";
(0,-20)*=0{}="4";
(40,10)*=0{}="5";
(50,10)*=0{}="6";
(60,10)*=0{}="7";
(50,0)*=0{\bullet}="8";
(70,0)*=0{\bullet}="9";
(60,-10)*=0{\bullet}="10";
(60,-20)*=0{}="11";
(13,0)*{\scriptstyle u};
(-7,-5)*{\scriptstyle b};
(7,-5)*{\scriptstyle c};
(3,-10)*{\scriptstyle v};
(-2,-15)*{\scriptstyle a};
(41,6)*{\scriptstyle d};
(48,6)*{\scriptstyle e};
(58,6)*{\scriptstyle f};
(53,-5)*{\scriptstyle b};
(67,-5)*{\scriptstyle c};
(63,-10)*{\scriptstyle v};
(58,-15)*{\scriptstyle a};
(73,0)*{\scriptstyle u};
(53,0)*{\scriptstyle w};
(30,-7)*{\partial_w};
(-10,-20)*=0{T/w};
(-10,-25)*=0{};
(52,-20)*=0{T};
"1";"3" **\dir{-};
"2";"3" **\dir{-};
"3";"4" **\dir{-};
"5";"8" **\dir{-};
"6";"8" **\dir{-};
"7";"8" **\dir{-};
"8";"10" **\dir{-};
"9";"10" **\dir{-};
"10";"11" **\dir{-};
{\ar(20, -10)*{};(40,-10)*{}};
\endxy
\label{outerface}
\end{equation}

Note that the possibility of removing the root vertex of $T$ is
included in this definition. This situation can happen only if the
root vertex is attached to exactly one inner edge, thus not every
tree $T$ has an outer face induced by its root. There is another
particular situation which requires special attention, that is the
inclusion of the tree with no vertices, called  the \emph{stump} and
denoted by $\eta$, to a tree with one vertex, called a
\emph{corolla}. In this case we get $n+1$ face maps if the corolla
has $n$ leaves. The operad $\Omega_p(\eta)$ consists of only one
colour and the identity operation on it. Then, a map of operads
$\Omega_p(\eta)\To \Omega_p(T)$ is just a choice of an edge of $T$.

\subsection{Degeneracy maps}
Suppose that $T$ is (a representative of) an object of $\rpd$ and
$v$ is a vertex of $T$ with valence one. Let $a$ be the sole
incoming edge and $b$ the outgoing edge of $v$. We obtain a new tree
$T\backslash v$ by removing $v$ from $T$ and identifying $a$ with
$b$ (in~(\ref{degeneracy}) below we refer to this new edge as
$\epsilon$). There is a map $\sigma_v:T \To T\backslash v$ in $\rpd$
associated to this operation which sends the colours $a$ and $b$ of
$\Omega_p(T)$ to $\epsilon$, sends the generating operation $v$ to
$\id_\epsilon$ and it is the identity for the other colours and
operations. The maps of this type are called the \emph{degeneracies}
of $T$.
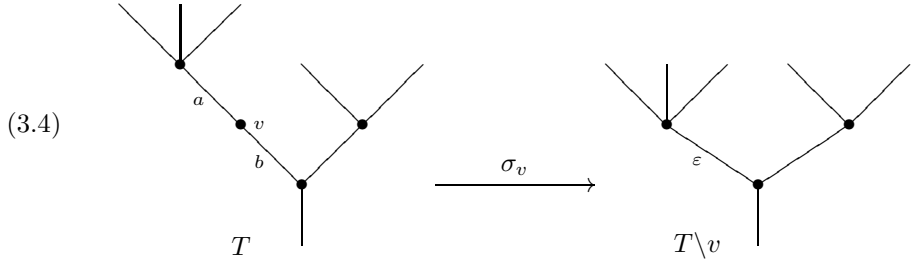
\begin{equation}
\xy<0.08cm,0cm>:
(0,20)*{}="1";
(10,20)*=0{}="2";
(20,20)*=0{}="3";
(10,10)*=0{\bullet}="4";
(20,0)*=0{\bullet}="5";
(30,-10)*=0{\bullet}="6";
(40,0)*=0{\bullet}="7";
(30,10)*=0{}="8";
(50,10)*=0{}="9";
(30,-20)*=0{}="10";
(80,10)*=0{}="11";
(90,10)*=0{}="12";
(100,10)*=0{}="13";
(90,0)*=0{\bullet}="14";
(105,-10)*=0{\bullet}="15";
(120,0)*=0{\bullet}="16";
(110,10)*=0{}="17";
(130,10)*=0{}="18";
(105,-20)*=0{}="19";
(13,4)*{\scriptstyle a};
(23,-6)*{\scriptstyle b};
(23,0)*{\scriptstyle v};
(95,-6)*{\scriptstyle \varepsilon};
(65,-7)*{\sigma_v};
"1";"4" **\dir{-};
"2";"4" **\dir{-};
"3";"4" **\dir{-};
"4";"5" **\dir{-};
"5";"6" **\dir{-};
"6";"7" **\dir{-};
"6";"10" **\dir{-};
"7";"8" **\dir{-};
"7";"9" **\dir{-};
"11";"14" **\dir{-};
"12";"14" **\dir{-};
"13";"14" **\dir{-};
"14";"15" **\dir{-};
"15";"16" **\dir{-};
"16";"17" **\dir{-};
"16";"18" **\dir{-};
"15";"19" **\dir{-};
{\ar(52, -10)*{};(78,-10)*{}};
(20,-20)*{T};
(95,-20)*{T\backslash v};
\endxy
\label{degeneracy}
\end{equation}

An important fact about faces and degeneracies is that they generate the maps in $\rpd$. Even more, we have the following
decomposition result, which is also a direct consequence of~\cite[Theorem 2.3.27]{Wei07}.
\begin{lem}\label{unique decomposition}
Every map $f:R\To T$ in $\rpd$ is either the identity or it decomposes uniquely as $f=d\circ s$, where $d$ is a composition
of face
maps and $s$ is a composition of degeneracies.
\end{lem}
\begin{proof}
To prove the existence of the decomposition, we proceed by induction
on $n=|V(R)|+|V(T)|$, the total number of vertices of $R$ and $T$.
If $n=0$ then $f:|\To |$ is the identity map and the statement is
obvious. In general, $f$ is a map of coloured operads and a part of
it consists of a map of sets between the colours, i.e., a map
$E(f)\colon E(R)\To E(T)$ between the edges. This map between the
sets of colours has a unique factorization as an epimorphism
followed by a monomorphism
$$
\xymatrix@1{E(R)\ar@{->>}[r]^-{s_0}&X\ \ar@{>->}[r]^-{d_0}&E(T)}.
$$
First, suppose that there exist $r_1\neq r_2\in E(R)$ such that
$s_0(r_1)=s_0(r_2)$. Since $f$ is a map of operads, $r_1$ and $r_2$
must be situated one above the other in a linear branch of $R$:
$$
\xy<0.1cm,0cm>:
(0, 15.5)*{\vdots};
(0, -19.5)*{\vdots};
(2,10)*{\scriptstyle r_1};
(2,-15)*{\scriptstyle r_2};
(0,12)*=0{}="2";
(0,6)*=0{\bullet}="3";
(0,0)*=0{}="4";
(0,-1.7)*=0{\vdots}="5";
(0,-6)*=0{}="6";
(0,-12)*=0{\bullet}="7";
(0,-18)*=0{}="8";
"2";"3" **\dir{-};
"3";"4" **\dir{-};
"6";"7" **\dir{-};
"7";"8" **\dir{-};
\endxy
$$
such that any edge $r$ between them satisfies $s_0(r)=s_0(r_1)=s_0(r_2)$. Hence we can suppose that $r_1$ and $r_2$ are
adjacent, joined by the vertex $v$:
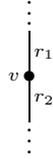
$$
\xy<0.1cm,0cm>:
(0, 9.5)*{\vdots};
(0, -7.5)*{\vdots};
(2,3)*{\scriptstyle r_1};
(2,-3)*{\scriptstyle r_2};
(-2,0)*{\scriptstyle v};
(0,6)*=0{}="1";
(0,0)*=0{\bullet}="2";
(0,-6)*=0{}="3";
"1";"2" **\dir{-};
"2";"3" **\dir{-};
\endxy
$$
It follows that $f$ decomposes as $R\stackrel{\sigma_v}{\twoheadrightarrow} R\backslash v\stackrel{f'}{\rightarrowtail}T$ and by the inductive
hypothesis we already have a decomposition $R\backslash v\twoheadrightarrow S\rightarrowtail T$ of $f'$.

Second, suppose that $s_0$ is bijective, hence we can assume that
$s_0$ is the identity map. If $ d_0$ is also the identity, it
follows that $f$ has to be the identity too. Indeed, since we are
working with non-symmetric operads, $f$ preserves the order of the
incoming edges at every vertex, and if $v\in\Omega_p(R)$ is a
generator (a vertex of $R$) and $f(v)$ is not a generator of
$\Omega_p(T)$ then there would be edges of $T$ without preimage in
$R$.

If $d_0$ is not the identity then let $e\in E(T)$ be an edge skipped by $d_0$. We can distinguish two cases:

If $e$ is an inner edge of $T$, it follows that $f$ decomposes as
$$\xymatrix@1{R\ar[r]^-{f'}&T/e\ \ \ar@{>->}[r]^-{\partial_e}&T}.$$
By the inductive hypothesis we obtain a decomposition of $f'$.

If $e$ is an outer edge of $T$, skipped by $d_0$. Since $f$ is a map
of operads, again any other outer edge adjacent to $e$ has to be
skipped by $d_0$. Denote the vertex adjacent to $e$ by $v$. It
follows that we can again decompose $f$ as
$$
\xymatrix@1{R\ar[r]^-{f'}&T/v\ \ \ar@{>->}[r]^-{\partial_v}&T}
$$
and obtain the desired factorization of $f$ by induction.

To prove the uniqueness of the decomposition we proceed in the
following way. Suppose that there are two factorizations of $f$:
$$
\xymatrix@1{R\ar@{->>}[r]^s&S\ \ar@{>->}[r]^d &T} \ \ \ \mbox{and}\ \ \  \xymatrix@1{R\ar@{->>}[r]^{s'}&S'\
\ar@{>->}[r]^{d'} &T}.
$$
Looking at the decompositions only on the level of the edges, it
follows that $E(S)=E(S')$, $\clr(s)=\clr(s')$ and
$\clr(d)=\clr(d')$. (Here $\clr(-)$ denotes the map at the level of
colours of the associated morphism of coloured operads.) Moreover,
since $s$ is a composition of degeneracies, it follows that if $v$
is a generator of $\Omega_p(R)$ then $s(v)=v$ or $s(v)$ is the
identity on some edge, in which case it is completely determined by
$\clr(s)$. Hence $s=s'$ and also $S=S'$. Similarly, $d$ is also
completely determined by what it does on the colours, thus $d=d'$.
\end{proof}

\begin{rem}
An extension of Lemma~\ref{unique decomposition} also holds, namely that the maps $d$ and $s$ also decompose uniquely into some naturally
ordered sequence of faces and degeneracies, respectively. In Section~\ref{generorders} we indicate how this is done.
\end{rem}

\begin{prop}\label{prop:allmonic}
Let $T$ be a tree in $\rpd$. Then, the faces of $T$ are exactly those injective operad maps $\Omega_p(R)\To\Omega_p(T)$
for which $|V(T)|=|V(R)|+1$ and the degeneracies of $T$ are exactly those surjective operad maps $\Omega_p(T)\To \Omega_p(S)$ for which
$|V(T)|=|V(S)|+1$.
\end{prop}
\begin{proof} We prove only the assertion for the faces, the other statement can be proved similarly.
Let $f\colon \Omega_p(R)\To \Omega_p(T)$ be an injective operad map with the required property and suppose that it is not a face.
By Lemma~\ref{unique decomposition} we know that $f$ can be decomposed as
\[
\xymatrix@1{R\ar[r]^s &R'\ar[r]^d &T,}
\]
where $s$ is a composite of degeneracies and $d$ is a composite of
faces. Note that $s$ cannot be the identity, since counting the
vertices would imply that $f=d$ is a face in that case. It follows
that $s$, and hence $f$ as well, are not injective on the edges.
This is a contradiction, since an injective operad map has to be
injective on the colours.
\end{proof}

\section{Dendroidal identities}\label{subsec:Dendroidal_identities}
In this section we are going to make explicit the relations between the generating maps (faces and degeneracies) of $\rpd$.
We present the relations in two different ways in this section, since both of these descriptions can be useful when reasoning with generators, as we will see later.

The first description uses the already familiar notation for faces
and degeneracies indexing the maps by edges and vertices of trees.
The second way, described in Section~\ref{generorders} is based on
natural linear orders defined on the set of faces and the set of
degeneracies of a given tree.

These relations, called the dendroidal relations, generalize the
simplicial identities in the category $\Delta$, henceforth we will
call them \emph{dendroidal identities}. The unique epi-mono
factorization theorem for maps in the category $\Delta$ extends to
the category $\rpd$, thus the relations we consider below cover
indeed all the cases: one only has to look at all possible
compositions $f=g_1\circ g_2$ of two generators of $\rpd$ and see
what are the other ways to decompose $f$ into two generators. The
result is summarized in Lemma \ref{dendroidal identities}.

We do not include in our first description the special case involving faces of the $n$-corolla, $n\ge 2$, although a statement similar to Lemma \ref{dendroidal identities} can be given.

There is a little ambiguity in the language that follows. For example $\partial_e$
can refer to two different face maps, but it is always clear from the context which one we are talking about.

\subsection{Elementary face relations}
\label{faceface}
Let $\partial_a:T/a\To T$ and $\partial_b:T/b\To T$ be two inner faces of $T$. It follows that the inner
faces $\partial_a:(T/b)/a\To T/b$ and $\partial_b:(T/a)/b\To T/a$ exist, $(T/a)/b=(T/b)/a$ and that the following diagram
commutes:
$$\xymatrix{(T/a)/b\ar[r]^-{\partial_b}\ar[d]_-{\partial_a} &T/a\ar[d]^-{\partial_a}\\
T/b\ar[r]^-{\partial_b} &T.}$$

Let $\partial_v:T/v\To T$ and $\partial_w:T/w\To T$ be two outer faces of $T$. Then the outer faces
$\partial_w:(T/v)/w\To T/v$ and $\partial_v:(T/w)/v\To T/w$ also exist, $(T/v)/w=(T/w)/v$ and the following diagram
commutes:
$$\xymatrix{(T/v)/w\ar[r]^-{\partial_w}\ar[d]_-{\partial_v} & T/v\ar[d]^-{\partial_v}\\ T/w\ar[r]^-{\partial_w} &T.}$$

The last remaining case is when we compose an inner face with an outer one in any order. There are several
possibilities, in all of them suppose that $\partial_v:T/v\To T$ is an outer face and $\partial_e:T/e\To T$ is an inner
face.
\begin{itemize}
\item[{\rm (i)}] If the inner edge $e$ is not adjacent in $T$ to the vertex $v$, then the outer face $\partial_v:(T/e)/v\To T/e$ and inner
face $\partial_e:(T/v)/e\To T/v$ exist, $(T/e)/v=(T/v)/e$ and the following diagram commutes:
$$\xymatrix{(T/v)/e\ar[r]^-{\partial_e}\ar[d]_-{\partial_v} &T/v\ar[d]^-{\partial_v}\\
T/e\ar[r]^-{\partial_e} &T.}$$
\item[{\rm (ii)}] Suppose that the inner edge $e$ is adjacent in  $T$ to the vertex $v$  and denote the other adjacent vertex to $e$ by
$w$. Following the notation of Section~\ref{the basic category}, $v$
and $w$ contribute to $T/e$ a vertex $v\circ_e w$ or $w\circ_e v$.
Let us denote this vertex  by $z$. Notice that the outer face
$\partial_z:(T/e)/z\To T/e$ exists if and only if the outer face
$\partial_w:(T/v)/w\To T/v$ exists and in this case
$(T/e)/z=(T/v)/w$. Moreover, the following diagram commutes:
$$\xymatrix{(T/v)/w\ar@{=}[r]\ar[d]_-{\partial_w} &(T/e)/z\ar[r]^-{\partial_z} &T/e\ar[d]^-{\partial_e}\\
T/v\ar[rr]^-{\partial_v} &&T.}$$ It follows that we can write
$\partial_v\partial_w=\partial_e\partial_z$ where $z=v\circ_e w$ if
$v$ is `closer' to the root of $T$ or $z=w\circ_e v$ if $w$ is
`closer' to the root of $T$.
\end{itemize}

\subsection{Elementary degeneracy relations}
\label{degdeg}
Let $\sigma_v:T \To T\backslash v$ and $\sigma_w:T\To T\backslash w$ be two degeneracies of $T$. Then the degeneracies
$\sigma_v:T\backslash w\To (T\backslash w)\backslash v$ and $\sigma_w:T\backslash v\To (T\backslash v)\backslash w$ exist,
$(T\backslash v)\backslash w = (T\backslash w)\backslash v$ and the following diagram commutes:
$$\xymatrix{T\ar[r]^-{\sigma_v}\ar[d]_-{\sigma_w} & T\backslash v\ar[d]^-{\sigma_w}\\
T\backslash w\ar[r]^-{\sigma_v} & (T\backslash v)\backslash w.}$$

\subsection{Combined relations}
Let $\sigma_v:T\To T\backslash v$ be a degeneracy and
$\partial:T'\To T$ a face map such that $\sigma_v:T'\To T'\backslash
v$ makes sense (i.e., $T'$ still contains $v$ and its two adjacent
edges as a subtree). Then, there exists an induced face map
$\partial:T'\backslash v\To T\backslash v$, determined by the same
vertex or edge as $\partial:T'\To T$. Moreover, the following
diagram commutes:
\begin{equation}
\xymatrix{T\ar[r]^-{\sigma_v}&T\backslash v\\
T'\ar[u]^-{\partial}\ar[r]^-{\sigma_v} &T'\backslash
v\ar[u]_-{\partial}.} \label{combrel01}
\end{equation}

Let $\sigma_v:T\To T\backslash v$ be a degeneracy and $\partial:T'\To T$ be a face map induced by one of the
adjacent edges to $v$ or the removal of $v$, if that is possible. It follows that $T'=T\backslash v$ and the composition
\begin{equation}
\xymatrix{ T\backslash v\ar[r]^-\partial &T\ar[r]^-{\sigma_v} &T\backslash v}
\label{combrel02}
\end{equation}
is the identity map $\id_{T\backslash v}$.

All these relations between the generators of the maps in $\rpd$ are
summarized in the following lemma whose proof is a direct
consequence of the dendroidal identities above.
\begin{lem}\label{dendroidal identities}
Let $f:R\To T$ be the composite of two generators (faces or
degeneracies) $f=g_1\circ g_2$, where both $R$ and $T$ have at least
one vertex. If $f\neq \id$, then there is exactly one more way to
write $f$ as the composition of two generators $f=g_1'\circ g_2'$,
where $\{g_1,g_2\}\ne\{g_1',g_2'\}$ as sets.  It follows that we
obtain a commutative diagram
$$
\xymatrix{R\ar[r]^{g_2}\ar[d]_{g_2'} &S\ar[d]^{g_1}\\
S'\ar[r]^{g_1'} &T}
$$
which is a special case of one of the diagrams of the dendroidal identities listed above.

If $f=\id$, then $g_1=\sigma_v$ for some vertex $v$ and
$g_2=\partial$ is one of the two possible face maps induced by an
edge adjacent to $v$ (or $v$ itself in some cases).  $\hfill\qed$
\end{lem}

\subsection{Linear orders on the faces and degeneracies}
\label{generorders}
Since the trees we consider are planar, we can canonically define a linear order on the set of all the faces of any chosen tree.
Similarly, a canonical linear order can be defined on the set of all degeneracies of a tree. We will treat the case of the
corollas separately.
There are a number of different possibilities to start with if one wants to obtain such orders, we choose the following.

Let $T$ be a tree in $\rpd$ such that $|V(T)|\ge 2$. Assign to each face of $T$ a natural number, respecting the following rules:
\begin{itemize}
\item[{\rm (i)}] If the vertex above the root $r\in V(T)$ is outer then assign the number 0 to~$\partial_r$.
\item[{\rm (ii)}] Starting from the root vertex, walk through all the edges and  vertices of $T$ by going always first to the left and upwards. When this is not possible any more, turn back to the closest, already visited vertex and choose the next, not yet covered edge left and upwards.
\item[{\rm (iii)}] Whenever an inner edge or an outer vertex is visited, assign the smallest not yet used natural number to the corresponding face of $T$.
\end{itemize}
Suppose that $T$ has $n$ face maps. The process described above defines a bijection
$$
\phi\colon \{\partial\,\mid\, \partial \textnormal{ is a face of $T$}\,\}\To \{0,1,\ldots,n-1\},
$$
hence also an order on the set of face maps of $T$. We define the $i$-th face of $T$ by $\partial_i:=\phi^{-1}(i)$.
For example, if $T$ is the tree

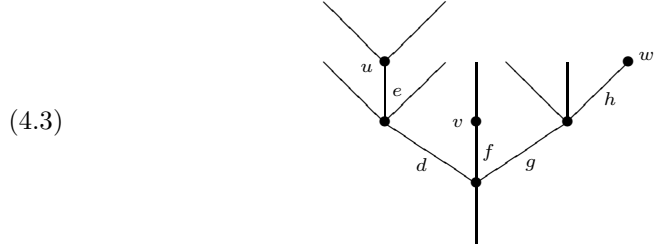
\begin{equation}\label{eq:tree1}
\xy<0.08cm,0cm>:
(25,-20)*=0{}="1";
(25,-10)*=0{\bullet}="2";
(10,0)*=0{\bullet}="3";
(25,0)*=0{\bullet}="4";
(40,0)*=0{\bullet}="5";
(0,10)*=0{}="6";
(10,10)*=0{\bullet}="7";
(20,10)*=0{}="8";
(25,10)*=0{}="9";
(30,10)*=0{}="10";
(40,10)*=0{}="11";
(50,10)*=0{\bullet}="12";
(0,20)*=0{}="13";
(20,20)*=0{}="14";
(7,9)*{\scriptstyle u};
(22,0)*{\scriptstyle v};
(53,11)*{\scriptstyle w};
(16,-7)*{\scriptstyle d};
(12,5)*{\scriptstyle e};
(27,-5)*{\scriptstyle f};
(34,-7)*{\scriptstyle g};
(47,4)*{\scriptstyle h};
"1";"2" **\dir{-};
"2";"3" **\dir{-};
"2";"4" **\dir{-};
"2";"5" **\dir{-};
"3";"6" **\dir{-};
"3";"7" **\dir{-};
"3";"8" **\dir{-};
"4";"9" **\dir{-};
"5";"10" **\dir{-};
"5";"11" **\dir{-};
"5";"12" **\dir{-};
"7";"13" **\dir{-};
"7";"14" **\dir{-};
\endxy
\end{equation}
then $T$ has eight faces and $\partial_0=\partial_d$, $\partial_1=\partial_e$, $\partial_2=\partial_u$, $\partial_3=\partial_f$, $\partial_4=\partial_v$, $\partial_5=\partial_g$, $\partial_6=\partial_h$ and $\partial_7=\partial_w$.

We can use this convention on traversing the tree $T$ to obtain another bijection
$$
\rho\colon\{\sigma\,\mid\, \sigma \textnormal{ is a degeneracy of $T$}\, \}\To\{0,1,\ldots, m-1\},
$$
provided $T$ has $m\ge 1$ degeneracies. For example, in the case of
the tree (\ref{eq:tree1}) drawn above $m=1$ and $\sigma_0=\sigma_v$.

In case $T$ is the $n$-corolla, the process of traversing $T$ from left to right induces a linear order on the set of faces of $T$ as well. After renaming these faces accordingly, we observe that  $\partial_0$ is the inclusion of the trivial tree $\eta$ into the root of $T$, $\partial_1$ is the inclusion of $\eta$ into the leftmost leaf of $T$, and so on.

\begin{rem}
The linear orders defined above extend the linear orders obtained from the usual numbering of faces and degeneracies in the simplicial category
$\Delta$. Indeed, if $T=L_n$ is the linear tree with $n$ vertices, then $\partial_i$ and $\sigma_i$ defined above correspond to the simplicial ones with the same index.
\end{rem}

One can ask wether the dendroidal identities remain the same as the simplicial ones with respect to the linear orders. This is certainly true for the elementary degeneracy relations. Indeed, after renaming the maps of any commutative diagram with degeneracies as in Section~\ref{degdeg}, the relation becomes $\sigma_j\sigma_i=\sigma_i\sigma_{j+1}$ for some $i\le j$.

On the other hand, the other types of elementary relations do not
remain valid. In the case of combined relations this fails because
there can be fewer degeneracies of a tree than faces. In the case of
elementary face relations, the tree pictured in (\ref{eq:tree1})
provides a counterexample since the relation
$\partial_f\partial_g=\partial_g\partial_f$ translates as
$\partial_3\partial_3=\partial_5\partial_3$. We can obtain other
counterexamples by considering the faces of the $n$-corolla. We
observe that such a situation can occur since some trees $T$ have
the property that the domain $R$ of a face $\partial\colon R\To T$
has two less faces than $T$. In general, there are two possibilities
for the elementary face relations:
$\partial_i\partial_{j-1}=\partial_{j}\partial_i$ or
$\partial_i\partial_{j-2}=\partial_{j}\partial_i$ when $i< j$.

\subsection{A sign convention for faces}
To any face map  $\partial: T\To R$ in $\rpd$ we can associate a sign $\sgn(\partial)\in \{\pm 1\}$. We begin by
numbering the vertices of $R$ from $0$ to $n$, starting with the root-vertex and
traversing the tree by going always first to the left. In this way we obtain a bijection $\sharp:V(R)\To \{0,\ldots,n\}$
(see the next picture for an example).
$$
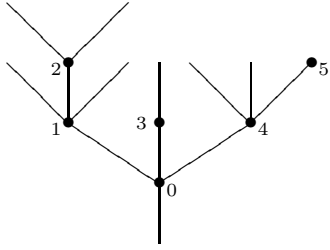
\xy<0.08cm,0cm>:
(25,0)*{}="1";
(25,10)*=0{\bullet}="2";
(10,20)*=0{\bullet}="3";
(25,20)*=0{\bullet}="4";
(40,20)*=0{\bullet}="5";
(0,30)*=0{}="6";
(10,30)*=0{\bullet}="7";
(20,30)*=0{}="8";
(25,30)*=0{}="9";
(30,30)*=0{}="10";
(40,30)*=0{}="11";
(50,30)*=0{\bullet}="12";
(0,40)*{}="13";
(20,40)*{}="14";
(27,9)*{\scriptstyle 0};
(8,19)*{\scriptstyle 1};
(8,29)*{\scriptstyle 2};
(22,20)*{\scriptstyle 3};
(42,19)*{\scriptstyle 4};
(52,29)*{\scriptstyle 5};
"1";"2" **\dir{-};
"2";"3" **\dir{-};
"2";"4" **\dir{-};
"2";"5" **\dir{-};
"3";"6" **\dir{-};
"3";"7" **\dir{-};
"3";"8" **\dir{-};
"4";"9" **\dir{-};
"5";"10" **\dir{-};
"5";"11" **\dir{-};
"5";"12" **\dir{-};
"7";"13" **\dir{-};
"7";"14" **\dir{-};
\endxy
$$
The sign of a face map $\partial$ is computed by the following rules:
\begin{itemize}
\item[{\rm (i)}] If $\partial$ is an inner face map induced by an edge $e$ and $v$ is the upper vertex adjacent to $e$, then
$\sgn(\partial)=(-1)^{\sharp v}$.
\item[{\rm (ii)}] If $\partial$ is an outer face map induced by the root-vertex $r$, then $\sgn(\partial)=(-1)^{\sharp r}$ which equals $1$.
\item[{\rm (iii)}] If $\partial$ is any other outer face map induced by a vertex $v$, then $\sgn(\partial)=(-1)^{\sharp v+1}$.
\end{itemize}

For example, if $\partial_b$ denotes the inner face map of Figure~\ref{innerface}, then $\sgn(\partial_b)=(-1)^1=-1$. If
$\partial_w$ denotes the outer face of Figure~\ref{outerface}, then $\sgn(\partial_w)=(-1)^2=1$.

There is one exception to these rules in the case of the inclusion of
the tree with no vertices into a corolla:
$$
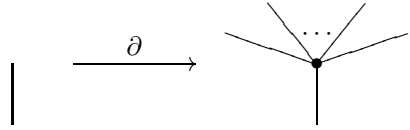
\xy<0.08cm,0cm>:
(0,0)*{}="1";
(0,10)*=0{}="2";
(50,0)*=0{}="3";
(50,10)*=0{\bullet}="4";
(35,15)*=0{}="5";
(42,20)*{}="6";
(58,20)*{}="7";
(65,15)*=0{}="8";
(20,13)*{\partial};
"1";"2" **\dir{-};
"3";"4" **\dir{-};
"4";"5" **\dir{-};
"4";"6" **\dir{-};
"4";"7" **\dir{-};
"4";"8" **\dir{-};
{\ar(10, 10)*{};(30,10)*{}};
(50,15)*{\cdots};
\endxy
$$
In this case, if $\partial$ takes the sole edge of the stump to the root of the corolla then $\sgn(\partial)=1$, otherwise
$\sgn(\partial)=-1$.

The following result is an immediate consequence of the elementary face relations (see Section~\ref{faceface}) and our
way of numbering the vertices.
\begin{lem}\label{decomp}
Let $f$ be a map in $\rpd$ such that it is a composition of two face maps $f=\partial_1\circ\partial_2$.
If $f=\partial'_1\circ\partial'_2$ is the other way to decompose $f$ as a composition of two faces then
$\sgn(\partial_1)\sgn(\partial_2)=-\sgn(\partial_1')\sgn(\partial_2').$
$\hfill\qed$
\end{lem}

\section{Normal faces}
For any tree $T$ a \emph{maximal linear part} of $T$ is an embedding $L_n\rightarrowtail T$
in $\rpd$ for some $n\ge 1$, such that whenever there is another such embedding $L_m\rightarrowtail T$ for $n\le m$ that fits into a commutative diagram of inclusions
$$
\xymatrix{L_n\ar@{>->}[r]\ar@{>->}[d]&T\\
L_m,\ar@{>->}[ru]}
$$
then $m=n$. We say that a face map $\partial\colon R\To T$ \emph{lives} or \emph{sits on a maximal linear part}
$L_n\rightarrowtail T$ when there exists another embedding $L_{n-1}\rightarrowtail R$ and a commutative diagram
\begin{equation}
\label{maxlinpart}
\xymatrix{L_{n-1}\ar@{>->}[r]\ar[d]_{\partial_i} &R\ar[d]^\partial\\
L_{n}\ar@{>->}[r] & T}
\end{equation}
for some face map $\partial_i:L_{n-1}\longrightarrow  L_n$, where the index $i$ is taken with respect to the order defined in
Section~\ref{generorders}. One can prove that if such a $\partial_i$ exists, then it is unique. Moreover, if $\partial$, $\partial'\colon R\To T$ sit on the same maximal linear part and $\partial_i$ fills diagram~(\ref{maxlinpart}) for both $\partial$ and $\partial'$, then $\partial=\partial'$.

We can observe that there are exactly $n+1$ faces sitting on a maximal linear part $\iota\colon L_n\To T$, and with respect to the linear order defined in Section \ref{generorders} they are the faces $\partial_k, \partial_{k+1},\ldots,\partial_{k+n}$ for some $k$. To underline the similarity between the faces of $[n]$ in the category $\Delta$, and the face maps sitting on a maximal linear part $\iota\colon L_n\To T$, it proves to be convenient to shift their indices so they become $\partial^{(\iota)}_0,\ldots, \partial^{(\iota)}_n$. We will also say that the faces living on
the maximal linear part are \emph{connected}.

\begin{dfn}
A face map $\partial\colon R\longrightarrow T$ in $\rpd$ is \emph{normal} if $\partial$ lives on a maximal linear part
$\iota\colon L_n\longrightarrow T$ for some $n\ge 1$ and $\partial=\partial_i^{(\iota)}$ for some $0\le i<n$ in the associated order.
\end{dfn}

\begin{exa}
Let $T$ and $R$ be the following trees:
$$
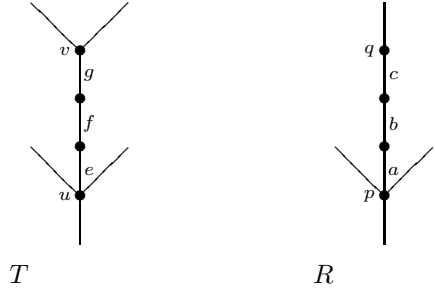
\xy<0.08cm,0cm>:
(0,0)*{
\xy
(-10,-5)*{T};
(0,0)*=0{}="1";
(0,8)*=0{\bullet}="2";
(-8,16)*=0{}="3";
(0,16)*=0{\bullet}="4";
(8,16)*=0{}="5";
(0,24)*=0{\bullet}="6";
(0,32)*=0{\bullet}="7";
(-8,40)*=0{}="8";
(8,40)*=0{}="9";
(-2.5,32)*{\scriptstyle v};
(-2.5,8)*{\scriptstyle u};
(1.5,28)*{\scriptstyle g};
(1.5,20)*{\scriptstyle f};
(1.5,12)*{\scriptstyle e};
"1";"2" **\dir{-};
"2";"3" **\dir{-};
"2";"4" **\dir{-};
"2";"5" **\dir{-};
"4";"6" **\dir{-};
"6";"7" **\dir{-};
"7";"8" **\dir{-};
"7";"9" **\dir{-};
\endxy
};
(50,0)*{
\xy
(-10,-5)*{R};
(0,0)*=0{}="1";
(0,8)*=0{\bullet}="2";
(-8,16)*=0{}="3";
(0,16)*=0{\bullet}="4";
(8,16)*=0{}="5";
(0,24)*=0{\bullet}="6";
(0,32)*=0{\bullet}="7";
(0,40)*=0{}="8";
(-2.5,32)*{\scriptstyle q};
(-2.5,8)*{\scriptstyle p};
(1.5,28)*{\scriptstyle c};
(1.5,20)*{\scriptstyle b};
(1.5,12)*{\scriptstyle a};
"1";"2" **\dir{-};
"2";"3" **\dir{-};
"2";"4" **\dir{-};
"2";"5" **\dir{-};
"4";"6" **\dir{-};
"6";"7" **\dir{-};
"7";"8" **\dir{-};
\endxy
};
\endxy
$$
The tree $T$ has one maximal linear part $L_2\To T$ and $R$ has one maximal linear part $L_3\To R$.
The faces of $T$ have the following properties: $\partial_e$ and $\partial_f$ are normal; $\partial_g$, $\partial_u$ and $\partial_v$ are not normal; $\partial_e$, $\partial_f$ and $\partial_g$ are connected to each other. The faces of $R$ have the following properties: $\partial_a$, $\partial_b$ and $\partial_c$ are normal; $\partial_p$ and $\partial_q$ are not normal; $\partial_a$, $\partial_b$, $\partial_c$ and $\partial_q$ are connected to each other.
\end{exa}

In general, if $\partial\colon R\To T$ is a face that lives on a maximal linear part $L_n\To T$ then $\partial$ is connected to
precisely $n$ other faces. Of these $n+1$ faces altogether, $n$ are normal and exactly one is not normal (the last one in
the induced order). A special case is that the face $\partial^{(\iota)}_0\colon L_0\To L_1$ is normal, while $\partial^{(\iota)}_1\colon L_0\To L_1$ is
not.

\begin{rem}
An arbitrary choice is made here about which faces to treat as
normal (i.e., exclude the case $i=n$). But we could have excluded
the case $i=0$ instead. For the general theory this does not make a
difference since if one makes the other choice then all the results
remain true with the obvious changes in the proofs.
\end{rem}

\section{Dendroidal complexes}
In this section we introduce the category of dendroidal complexes
$\dCh$. This category extends the category of chain complexes of
abelian groups and, as we will see in the next section, it is
equivalent to the category of dendroidal abelian groups $\dAb$,
i.e., the category of functors from $\rpd^{{\rm op}}$ to abelian
groups. If $A$ is a dendroidal abelian group and $f\colon
R\longrightarrow T$ is any map in $\rpd$, then the associated group
homomorphism $A_T\longrightarrow A_R$ is denoted by $f^*$.

We say that an abelian group $A$ is an $\rpd$-graded abelian group if $A=\bigoplus_{T\in \rpd} A_T$.
\begin{dfn}
A \emph{dendroidal complex} $A$ is an $\rpd$-graded abelian group together with structure maps given by group
homomorphisms $\delta^\sharp\colon A_T\To A_R$, for every face map $\delta\colon R\To T$, satisfying the
following two conditions:
\begin{itemize}
\item[{\rm (i)}] If $\delta$ is a normal face, then $\delta^\sharp=0$.
\item[{\rm (ii)}] For any commutative diagram of elementary face relations
$$
\xymatrix{
S\ar[r]^\partial\ar[d]_\delta & R \ar[d]^{\partial_1} \\
R'\ar[r]^{\delta_1} & T }
$$
the associated diagram
\[
\xymatrix{
A_S & A_{R} \ar[l]_{\partial^\sharp} \\
A_{R'}\ar[u]^{\delta^\sharp} & A_T\ar[l]_{\delta_1^\sharp}\ar[u]_{\partial_1^\sharp} }
\]
anticommutes, i.e., $\delta^{\sharp}\delta_1^{\sharp}=-\partial^{\sharp}\partial_1^{\sharp}$.
\end{itemize}
\end{dfn}

A map $\varphi\colon A\longrightarrow B$ between dendroidal complexes is given by a sequence of maps $\varphi_T\colon A_T\longrightarrow B_T$
compatible with the group homomorphisms. We denote the category of dendroidal complexes by $d\Ch$. Note
that if in the definition of a dendroidal complex we replace the category $\rpd$ by its full subcategory $\Delta$,
then we recover the notion of a chain complex. In fact, if we denote by $\Ch$ the category of chain complexes, then there is a pair of
adjoint functors
\begin{equation}
\xymatrix{
j_!\colon \Ch \ar@<3pt>[r] & \ar@<3pt>[l] \dCh\colon j^*
}
\label{dChadj}
\end{equation}
where $j^*$ is the restriction functor, i.e., $j^*(A)_n=A_{L_n}$. Its left adjoint $j_!$ is `extension by zero' and sends a chain complex
$A$ to the dendroidal complex
$$
j_!(A)_T=\left\{
\begin{array}{cc}
A_n & \mbox{ if $T\cong L_n$}, \\
\emptyset & \mbox{ otherwise}.
\end{array}
\right.
$$

\subsection{The Moore dendroidal complex}
Let $A$ be a dendroidal abelian group. One can define a dendroidal complex $MA$ associated to $A$ by
setting $(MA)_T=A_T$ for every $T\in \rpd$. For any face $\delta\colon R\To T$ in $\rpd$, the structure map $\delta^\sharp$ is defined as
follows:
\begin{itemize}
\item[{\rm (i)}] If $\delta$ is a normal face, then $\delta^\sharp=0$.
\item[{\rm (ii)}] If $\delta$ is not a normal face and it is not connected to any normal face, then $\delta^\sharp=\sgn{(\delta)}\cdot\delta^*$.
\item[{\rm (iii)}] In the remaining case $\delta=\partial^{(\iota)}_n$ for a maximal linear part
$L_n\stackrel{\iota}{\longrightarrow} T$ in the induced order. Define in this case
\[
\delta^\sharp=\sum_{i=0}^n\sgn(\partial_i)\cdot\partial_i^*.
\]
\end{itemize}
\begin{lem}
The $\rpd$-graded abelian group $MA$ defined above is in fact a dendroidal complex.
\end{lem}
\begin{proof} Suppose that
\[
\xymatrix{
S\ar[r]^{\delta'}\ar[d]_\delta & R \ar[d]^{\delta'_1} \\
R'\ar[r]^{\delta_1} & T }
\]
is a commutative diagram of elementary face relations. There are several cases to distinguish according to the type of faces involved.

If none of the four faces are normal or connected to a normal face
then Lemma~\ref{decomp} ensures that the induced square
anticommutes.

If each of the sets $\{\delta,\delta_1\}$, $\{\delta', \delta'_1\}$ contains at least one normal face then the induced square anticommutes trivially, since at least one of the induced maps in each set is the zero map.

In case the set $\{\delta',\delta'_1\}$ contains a normal face while $\{\delta, \delta_1\}$ does not contain any, one has to prove that $\delta^\sharp\delta_1^\sharp=0$. First we observe that if $\delta'$ is normal, then $\delta_1$ is also normal, hence we can assume that
$\delta_1'$ is the only normal face in the diagram. Moreover, it also follows that $\delta_1'=\partial_{n-1}$ in the order induced by a maximal linear part $L_n\To T$ and, since none of the elements of $\{\delta,\delta_1\}$ can be normal faces, $\delta'$ is connected to $\delta_1'$, and both $\delta$ and $\delta_1$ are connected to normal faces. Hence
\[
\delta^\sharp=\sum_{i=0}^n\sgn(\partial_i)\cdot\partial_i^*\quad\mbox{and}\quad
\delta_1^\sharp=\sum_{i=0}^{n+1}\sgn((\partial_1)_i)\cdot(\partial_1)_i^*.
\]
We now conclude that $\delta^\sharp\delta_1^\sharp=0$ by following the
proof for the simplicial case, when one proves the differential
property $d^2=0$ for the Moore chain complex associated to a
simplicial abelian group.

The remaining case, when none of the four faces are normal but some of them are connected to normal ones, brakes down into the following three cases.
\begin{itemize}
\item[{\rm (i)}] When both of $\delta'$ and $\delta_1'$ are not normal, but connected to normal faces, they have to live on different maximal linear parts of the tree $T$. Hence $\delta$ and $\delta_1$ are not  normal, but connected to normal faces. Moreover, each summand in the definition of $ \delta^\sharp$ and $\delta'^\sharp$ fits into a commutative diagram
\[
\xymatrix{A_S &A_R\ar[l]_-{(\partial')_j^*}\\
A_{R'}\ar[u]^{\partial_i^*} &A_T\ar[l]^-{(\partial_1)_j^*}\ar[u]_{(\partial_1')_i^*}}
\]
hence Lemma \ref{decomp} ensures that the required square anticommutes.
\item[{\rm (ii)}] Suppose that $\delta_1'$ is connected to normal faces, while $\delta'$ is not. We analyze the case when $\delta'$ is `adjacent'
to $\delta_1'$ (the other cases are easier and we omit them). This
situation can typically be illustrated when $T$ is a tree of the
form
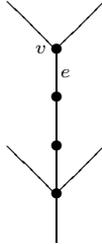
\[
\xy<0.08cm,0cm>:
(0,0)*=0{}="1";
(0,8)*=0{\bullet}="2";
(-8,16)*=0{}="3";
(0,16)*=0{\bullet}="4";
(8,16)*=0{}="5";
(0,24)*=0{\bullet}="6";
(0,32)*=0{\bullet}="7";
(-8,40)*=0{}="8";
(8,40)*=0{}="9";
(-2.5,32)*{\scriptstyle v};
(1.5,28)*{\scriptstyle e};
"1";"2" **\dir{-};
"2";"3" **\dir{-};
"2";"4" **\dir{-};
"2";"5" **\dir{-};
"4";"6" **\dir{-};
"6";"7" **\dir{-};
"7";"8" **\dir{-};
"7";"9" **\dir{-};
\endxy
\]
and $\delta_1'=\partial_e$, $\delta_1=\partial_v$. Then it is clear what the other trees and face maps are in the diagram, and we see that $\delta\colon S\To R'$ is connected to normal faces, but is not normal. Again, a summand $(\partial_1')^*_i$ of $(\delta_1')^\sharp$ will correspond to the summand $\partial_i^*$ of $\delta^\sharp$ such that the associated diagram
\[
\xymatrix{A_S &A_R\ar[l]_-{(\delta')^*}\\
A_{R'}\ar[u]^{\partial_i^*} &A_T\ar[l]^-{(\delta_1)^*}\ar[u]_{(\partial_1')_i^*}}
\]
commutes for every $i$. We conclude that the required diagram is anticommutative.
\item[{\rm (iii)}] The remaining case, when $\delta'$ is connected to normal faces and $\delta_1'$ is not, is symmetric to case (ii).
\end{itemize}
This exhausts all the possible combinations of faces and completes the proof.
\end{proof}
The dendroidal complex $MA$ is called the \emph{Moore complex} associated to $A$.

\subsection{The normalized and degenerate dendroidal subcomplexes}
Let $A$ be a dendroidal abelian group. We can construct a subcomplex
$NA$ of the Moore complex $MA$ by setting
\[
(NA)_T=\cap_{\tilde\partial}\ker(\tilde\partial^*)\leq A_T,
\]
where $\tilde\partial$ runs through all normal faces with codomain
$T$. Note that if $T$ has no normal faces, then we have an empty
intersection and in that case we set $(NA)_T=A_T$. We can restrict
the structure maps of $MA$ to get a dendroidal complex structure on
$NA$. The dendroidal complex $NA$ is called the \emph{normalized
complex} associated to $A$.

Another  subcomplex $DA$ of $MA$ is defined by
\[
(DA)_T=\!\!\!\!\!\sum_{\sigma\colon T\rightarrow S}\!\!\!\!\sigma^*(A_S)\leq A_T,
\]
where $\sigma$ runs through all degeneracies with domain $T$. Again, the structure maps of $MA$
restrict to $DA$, thus we obtain a dendroidal subcomplex of $MA$ which is called the \emph{degenerate
complex} associated to $A$.

\begin{prop}
The $\rpd$-graded abelian groups $NA$ and $DA$ are dendroidal subcomplexes of $MA$.
\end{prop}
\begin{proof}
We first consider the case of $NA$. Let $\delta\colon R\To T$ be a face map in $\rpd$. We need to prove that
every $x\in (NA)_T$ satisfies $\delta^\sharp(x)\in (NA)_R$. There are three cases to distinguish:
\begin{itemize}
\item[{\rm (i)}] If $\delta$ is a normal face, then $\delta^\sharp(x)=0\in (NA)_R$.
\item[{\rm (ii)}] If $\delta$ is neither normal nor connected to a normal face, then
$\delta^\sharp=\sgn(\delta)\cdot\delta^*$. Suppose that $\partial\colon S\To R$ is a normal face. There is a
commutative diagram of elementary face relations
\[
\xymatrix{S\ar[r]^-{\partial}\ar[d]_{\tilde\gamma} &R\ar[d]^\delta\\
R'\ar[r]^{\gamma} &T}
\]
by Lemma~\ref{dendroidal identities}. It is easy to check that in such a case whenever $\partial$ is normal, $\gamma$
is normal as well. We conclude that
\[
\partial^*\delta^*(x)=\tilde\gamma^*\gamma^*(x)=0
\]
and thus $\partial^*\delta^\sharp(x)=0$.
\item[{\rm (iii)}] In the remaining case, $\delta=\partial_n$ for some maximal linear part $\iota\colon L_n\To T$. Therefore
\[
\delta^\sharp=\sum_{i=0}^n\sgn(\partial_i)\cdot\partial_i^*.
\]
Again let $\partial\colon S\To R$ be a normal face. In the same way as in case (ii), for every $i$ we have
$\partial_i\partial=\gamma_i\tilde\gamma_i$ for some normal face $\gamma_i$. Hence every summand of
$\partial^*\delta^\sharp$ vanishes on $x$.
\end{itemize}
Next, we prove that $DA$ is a subcomplex. Suppose that $\delta\colon R\To T$ is a face map and let $x\in (DA)_T$ such that
$x=\sigma_1^*(x_1) +\ldots+\sigma_k^*(x_k)$ for some degeneracies $\sigma_j\colon T\To S_j$ and $x_j\in A_{S_j}$. There are
again three cases to distinguish:
\begin{itemize}
\item[{\rm (i)}] If $\delta$ is a normal face, then $\delta^\sharp(x)=0\in (DA)_R$ as before.
\item[{\rm (ii)}]If $\delta$ is neither normal nor connected to a normal face, then
there exists a commutative diagram of combined dendroidal relations
\[
\xymatrix{
T\ar[r]^{\sigma_j}&S_j\\
R\ar[r]^{\sigma_j'}\ar[u]^{\delta} &T_j\ar[u]_{\delta_j} }
\]
for every $j$ (otherwise $\delta$ would be a section of a degeneracy, hence normal or connected to a normal face). In this case
\[
\delta^*\sigma_j^*(x_i)=\sigma_j'^*\delta_j^*(x_j)\in (DA)_R
\]
for every $j$, thus $\delta^\sharp(x)\in (DA)_R$.
\item[{\rm (iii)}] In the remaining case, $\delta^\sharp=\sum_{i=0}^n\sgn(\partial_i)\cdot\partial_i^*$ where $\partial_0,\ldots,\partial_{n-1}$ are normal faces sitting on the same maximal linear part $\iota\colon L_n\To T$ where the indices come from the induced order, and $\delta=\partial_n$ in this order. It follows that
\[
\delta^\sharp(x)=\sum_{i,j}\sgn(\partial_i)\partial_i^*\sigma_j^*(x_j).
\]
This sum can be divided into two parts. The first part consists of
those components for which $\sigma_j\partial_i$ satisfies the
combined dendroidal relation~(\ref{combrel01}), i.e.,
$\sigma_j\partial_i=\partial'_i\sigma'_j$ for some face
$\partial'_i$ and degeneracy $\sigma'_j$. This part of the sum is
clearly in $(DA)_R$. The second part consists of those summands for
which $\sigma_j\partial_i$ satisfies the combined dendroidal
relation~(\ref{combrel02}) of the second type, that is
$\sigma_j\partial_i=\sigma_j\partial'_i=\id_R$ for some face
$\partial_i'$. But in such a case
$\sgn(\partial_i')=-\sgn(\partial_i)$ and one can form such pairs
from the components of this part of the sum cancelling each other.
\end{itemize}
\end{proof}

The Moore dendroidal complex associated to a dendroidal abelian group splits as a direct sum of the normalized part and the degenerate part. The approach is similar to the one appearing in \cite{weibel}, which establishes the same property for the classical Moore complex of a simplicial abelian group. We need to show that $A_T=(NA)_T\oplus (DA)_T$ for every tree $T\in\rpd$.

\begin{lem}\label{direct sum 1}
For any dendroidal abelian group $A$, the dendroidal complexes $NA$ and $DA$ satisfy that $(NA)_T\cap (DA)_T={0}$ for every tree $T\in\rpd$.
\end{lem}
\begin{proof}
Suppose that $0\neq x\in (NA)_T\cap(DA)_T$ and write $x$ as a finite sum of degeneracies
\[
x=\sigma_1^*(x_1)+\cdots+\sigma_k^*(x_k),
\]
such that the number of the summands is minimal. If $k=1$ then $\sigma_1\colon T\To S$ has two right inverses in $\rpd$ and at
least one of them, say $\partial\colon S\To T$, is a normal face. It follows that $0=\partial^*(x)=(\sigma_1\partial)^*(x_1)=x_1$
which contradicts $x\neq 0$.

If $k>1$, we can use a similar argument.  Since $k$ is minimal, $\sigma_i\neq \sigma_j$ for every $i\ne j$, hence $\sigma_i$ and $\sigma_j$ are induced by univalent vertices
$v_i\neq v_j$. We can suppose that $\sigma_1$ sits on a linear part $L_n\To T$ and that all the other $\sigma_i$ are on a different linear part or, if on the same one, that they come after $\sigma_1$ in the induced order. In other words, none of those vertices $v_2,\ldots,v_k$ which are on the linear component of $v_1$ sit below $v_1$. Let $\partial$ be the normal right inverse to $\sigma_k$ induced by the edge below $v_1$ or by cutting $v_1$. Then
\[
\partial^*(x)-\big(\sigma_2\partial)^*(x_2)+\cdots+(\sigma_{k}\partial)^*(x_{k})\big)=x_1,
\]
\[
x=\sigma^*_2(x_2)+\cdots+\sigma^*_{k}(x_{k})-\big((\sigma_2\partial\sigma_1)^*(x_2)+\cdots+(\sigma_{k}
\partial\sigma_1)^*(x_{k})\big).
\]
Let us look at the composite $\sigma_i\partial\sigma_1$ for all $i>1$ and write it in another form with the help of the dendroidal identities. We observe that $\sigma_i\circ\partial\ne \id_T$ since we chose $\partial$ in a way that avoids this situation.
It follows that we obtain commutative diagrams for all $1<i\le k$:
\[
\xymatrix{T\ar[r]^-{\sigma_1}\ar@{.>}[d]_-{\sigma_i}& T\backslash v_1\ar[r]^-\partial\ar[d]_{\sigma'} &T\ar[d]^{\sigma_i}\\
T\backslash v_i\ar@{.>}[r]^-{\sigma''} &(T\backslash v_1)\backslash v_i \ar[r]^-{\partial'} &T\backslash v_i{\rlap ,}}
\]
where by the dendroidal identities the dotted vertical arrow can only be $\sigma_i$. We conclude that $(\sigma_i\partial\sigma_1)^*(u)=\sigma_i^*(u')$
for all $1<i\le k$ and
\[
x=\sigma_2^*(y_2)+\cdots+\sigma_{k}^*(y_{k}).
\]
This is a contradiction since $k$ was chosen to be minimal.
\end{proof}

\begin{lem}\label{direct sum 2}
For any dendroidal abelian group $A$, the dendroidal complexes $NA$ and $DA$ satisfy that $(NA)_T+(DA)_T=A_T$ for every tree $T\in\rpd$.
\end{lem}
\begin{proof}
Fix an $x\in A_T$ and define
\[
N_x=\big\{\partial\colon S\To T\ \big|\ \partial \mbox{ is a normal face such that } \partial^*(x)\neq 0\big\}.
\]
We can assume that $N_x$ is not empty, otherwise $x\in (NA)_T$. Let
$k\in \N$ be the number of the maximal linear parts of $T$.
Partition $N_x$ into subsets $N_x=N_x^{(\iota_1)}\cup\cdots\cup
N_x^{(\iota_k)}$ for every maximal linear part $\iota_j\colon
L_{n_j}\To T$, where $N_x^{(\iota)}$ contains those elements of
$N_x$ which sit on $\iota$.

The goal is to write $x$ as $x=x_1+y_1$, where $y_1\in (DA)_T$ and $N_{x_1}^{(\iota_1)}=\emptyset$, while $|N_{x_1}^\iota|\leq |N_x^\iota|$ for the other linear parts $\iota\ne \iota_1$. If we succeed, we can iterate the process by finding a decomposition $x_1=x_2+y_2$ such that $x_2$ kills the set $N_{x_1}^{(\iota_2)}$ while it does not increase the size of the other sets $N_{x_1}^{(\iota)}$.
After $k$ steps we would arrive at a decomposition
\[
x=x_k+y_1+y_2+\cdots+y_k,
\]
where $N_{x_k}=\emptyset$, thus $x_k\in (NA)_T$ and $y_1+\cdots+y_k\in (DA)_T$, thereby finishing the proof.

To obtain such a decomposition of $x$, we proceed as follows. Let $\partial\in N_x^{(\iota_1)}$
be the smallest element in the order induced by $\iota_1\colon L_{n_1}\To T$. Let $\partial^*(x)=y$ and define $x'=x-\sigma^*(y)$ where $\sigma\colon T\to S$ is the
biggest such degeneracy in the order induced by $\iota_1$ for which $\sigma\partial=\id_{S}$.
It follows that
\[
\partial^*(x')=\partial^*(x)-(\sigma\partial)^*(y)=0.
\]
Now suppose that $\tilde\partial$ is any normal face such that $\tilde\partial^*(x)=0$. We are going to prove that $\tilde\partial^*(x')\neq 0$ can happen only if $\tilde\partial$ is connected to $\partial$ and $\tilde\partial> \partial$ in the order induced by $\iota_1$.
Indeed, on the one hand if $\tilde\partial$ is not connected to $\partial$ then it is obvious that $\sigma\tilde\partial\ne\id$. On the other hand, the reason we chose $\partial$ to be minimal and $\sigma$ maximal was that now $\sigma\tilde\delta\ne\id$ holds also whenever $\tilde\partial<\partial$ on the linear part $\iota_1$. Hence if $\tilde\partial$ obeys one of these cases, we can fill the following diagram of dendroidal identities
\[
\xymatrix{T\ar[r]^-{\sigma} &T\backslash v\ar[r]^-{\partial} &T\\
T'\ar[u]^-{\tilde\partial}\ar[r]^-{\sigma'} &T'\backslash
v\ar[u]^-{\partial'}\ar@{.>}[r]^-{\partial''} &T'{\rlap .}\ar@{.>}[u]_-{\tilde\partial}}
\]
It is immediate that in this diagram the dotted vertical arrow is $\tilde\partial$. Therefore
\[
\tilde\partial^*(x')=-\tilde\partial^*\sigma^*(y)=-\tilde\partial^*\sigma^*\partial^*(x)=-\sigma'^*\partial''^*
\tilde\partial^*(x)=0.
\]
Let us summarize what we managed to achieve with the $x=x'+\sigma(y)$ decomposition:
\begin{itemize}
\item[{\rm (i)}] The normal face $\partial$ satisfies $\partial(x)\ne 0$ and $\partial(x')=0$.
\item[{\rm (ii)}] For any normal face $\tilde\partial$ such that $\tilde\partial^*(x)=0$, we can have $\tilde\partial^*(x')\neq 0$ only if $\tilde\partial> \partial$ on the same linear part $\iota_1$.
\end{itemize}
Now we can apply the same process for $x'$ and the new smallest element in $N_{x'}^{(\iota_1)}$, and so on. In a finite number of steps we arrive to the desired decomposition $x=x_1+y_1$.
\end{proof}

\begin{prop}\label{direct sum}
For any dendroidal abelian group $A$ the associated Moore dendroidal complex decomposes as $MA=NA\oplus DA$.
\end{prop}
\begin{proof}
It follows directly from Lemma~\ref{direct sum 1} and Lemma~\ref{direct sum 2}.
\end{proof}

\section{The Dold-Kan correspondence}
The canonical inclusion $i\colon\Delta\longrightarrow \rpd$ induces a restriction functor $i^*\colon \dAb\longrightarrow s\A$
which has a left adjoint $i_!$ given by Kan extension. The functor $i^*$ sends
a dendroidal abelian group $A$ to the simplicial abelian group
$$
i^*(A)_n = A_{i([n])}.
$$
Its left adjoint $i_!\colon s\A \longrightarrow \dAb$ is `extension
by zero', and sends a simplicial abelian group $A$ to the dendroidal
abelian group given by
$$
i_!(A)_T=\left\{
\begin{array}{cc}
A_n & \mbox{ if $T\cong i([n])$}, \\
\emptyset & \mbox{ otherwise}.
\end{array}
\right.
$$
We will define a right adjoint $\Gamma$ to the normalized dendroidal
complex functor $N\colon\dAb\longrightarrow \dCh$ and we will prove
that the pair $(N,\Gamma)$ forms an equivalence of categories. This
equivalence extends the classical Dold-Kan correspondence for
simplicial abelian groups and chain complexes in the following
precise sense. There is a commutative diagram of adjoint functors
\[
\xymatrix@R+3mm@C+3mm{ s\A\ar@<2pt>[r]^-{i_!}\ar@<-2pt>[d]_{N_s} &\dAb\ar@<-2pt>[d]_N\ar@<2pt>[l]^-{i^*} \\
\Ch\ar@<2pt>[r]^-{j_!}\ar@<-2pt>[u]_{\Gamma_s}&\dCh{\rlap ,}\ar@<2pt>[l]^-{j^*}\ar@<-2pt>[u]_{\Gamma}}.
\]
where $(j_!, j^*)$ is the adjunction described in~(\ref{dChadj}). Moreover the following relations hold:
\[
\begin{array}{ccccccc}
N_s i^*=j^*N, &&&& Ni_!=j_!N_s,\\
\Gamma_s j^*=i^*\Gamma, &&&& \Gamma j_!=i_!\Gamma_s,\\
i^*i_!=\id, &&&& j^*j_!=\id, \\
N\Gamma=\id, &&&& N_s\Gamma_s=\id,\\
\Gamma N \cong \id, &&&& \Gamma_s N_s\cong \id.
\end{array}
\]
The rest of this section is devoted to the construction of the
functor $\Gamma$. We denote by $\Omega_{mono}$ the subcategory of
$\rpd$ consisting of all the trees as objects and only monomorphisms
as maps. For every dendroidal complex $C$ there is a functor
$F_C\colon \Omega_{mono}^{\rm op}\To\A$ defined on objects by
$F_C(T)=C_T$ and on face maps by
\[
F_C(\partial)=
\begin{cases}
0 & \textnormal{if $\partial$ is normal,}\\
\sgn(\partial)\partial^\sharp & \textnormal{oherwise.}
\end{cases}
\]
Observe that $F_C$ is indeed a functor since the sign convention on
faces implies that commutative diagrams of dendroidal identities
involving faces are taken via $F_C$ to commutative diagrams of
abelian groups. This functor will play a role in the construction of
a right adjoint $\Gamma\colon \dCh\To \dAb$ to the normalized
dendroidal complex functor $N$.

If we assume that such a right adjoint exists, then there should be
a one-to-one correspondence between the sets of morphisms
\[
\dCh(NA,C)\cong \dAb(A, \Gamma C)
\]
for any dendroidal complex $C$ and dendroidal abelian group $A$. If one takes $A$ to be the representable $\Z\rpd[T]$, then
\begin{equation}\label{gamma}
\dCh(N\Z\rpd[T], C)\cong\dAb(\Z\rpd[T],\Gamma C)\cong (\Gamma C)_T
\end{equation}
by the Yoneda lemma. Moreover, this correspondence has to be an isomorphism of groups, showing us a way to define $(\Gamma C)_T$ for every tree $T$.
One can unpack the left hand side of equation (\ref{gamma}) to arrive at the definition
\[
(\Gamma C)_T=\bigoplus_{r\colon T\twoheadrightarrow R}C_R,
\]
where $r$ runs through all epimorphisms in $\rpd$ with domain $T$. In the direct sum above we will denote by $C_R^r$ the component $C_R$ corresponding to an epimorphism~$r$.

We still have to define $\Gamma C$ on the maps of $\rpd$. Suppose that $f\colon S\To T$ is such a map and define $f^*\colon(\Gamma C)_T\To (\Gamma C)_S$ in the following way. Let $r\colon T\twoheadrightarrow R$ be an epimorphism in $\rpd$. The map $r\circ f\colon S\To R$ has a unique factorization $d\circ s$ by Lemma~\ref{unique decomposition}:
\[
\xymatrix{S\ar@{>>}[r]^s\ar[d]_f &S'\ar@{>->}[d]^d\\
T\ar@{>>}[r]^r & R{\rlap .}}
\]
We define $f^*$ on the component $C_R^r$ as the composite
\[
\xymatrix@C+5mm{(f^*)^r\colon C_R^r\ar[r]^-{F_C(d)} &C_{S'}^s\ar@{>->}[r] &(\Gamma C)_S}.
\]
We have finished the definition of $\Gamma$ on objects. Let us check that $\Gamma C$ is indeed a dendroidal abelian group for every $C\in \dCh$. It is easy to see that $\Gamma C(\id_T)=\id\colon \Gamma C_T\To \Gamma C_T$. Suppose that $f\colon S\To T$ and $g\colon U\To S$ are two maps in $\rpd$. We need to check that for any epimorphism $r\colon T\twoheadrightarrow R$, the components $((fg)^*)^r$ and $(g^*f^*)^r$ are the same. Indeed, since the epi-mono factorizations of $rfg$, and of $rf$ followed by $sg$ are unique, we infer that $d=d_fd_g$ in the following diagram.
\[
\xymatrix{U\ar[d]_g\ar@{>>}[r]^u &U'\ar@/^6mm/[dd]^d\ar@{>->}[d]^{d_g}\\
S\ar[d]_f\ar@{>>}[r]^s&S'\ar@{>->}[d]^{d_f}\\
T\ar@{>>}[r]^r &R}
\]
Since $F_C$ is a functor, this implies the required equality.

It is easy to check that the obvious definition of $\Gamma$ on maps of dendroidal complexes is functorial. Now we can prove the following propositions.

\begin{prop}\label{counit}
For every tree $T\in\rpd$ the abelian groups  $(N\Gamma C)_T$ and $C_T$ are equal.
\end{prop}
\begin{proof}
We have two decompositions of the abelian group $(\Gamma C)_T$ into a direct sum of subgroups. First, by definition
\[
(\Gamma C)_T=C_T^{\id_T}\oplus \mathop{\bigoplus_{T\stackrel{r}{\twoheadrightarrow}R}}_{r\ne\id_T} C_R^r
\]
and second, by Proposition \ref{direct sum}
\[
(\Gamma C)_T=(N\Gamma C)_T\oplus (D\Gamma C)_T.
\]
Hence it is enough to prove that $\bigoplus_{r\ne\id_T} C_R^r
\le (D\Gamma C)_T$ and $C_T^{\id_T}\le (N\Gamma C)_T$.

To see the first assertion we pick an epimorphism $r\colon T\twoheadrightarrow R$, $r\ne \id_T$ and prove that the corresponding component $C_R^r\le (\Gamma C)_T$ is in the image of a degeneracy. From our choice it follows that $r$ decomposes as $r=\sigma\circ r'$ where $\sigma\colon T\To S$ is a degeneracy and $r'\colon S\twoheadrightarrow R$ is another epimorphism (possibly the identity). Let us look at the image of $\sigma^*\colon(\Gamma C)_S\To (\Gamma C)_T$ on the component $C_R^{r'}$. Since the unique epi-mono factorization of $r'\sigma$ is
\[
\xymatrix{T\ar@{>>}[r]^r\ar[d]_\sigma &R\ar@{=}[d]\\
S\ar@{>>}[r]^{r'} &R{\rlap ,}}
\]
we can conclude that $\sigma^*$ sends the component $C_R^{r'}$ to the component $C_R^r$.

The second assertion follows as well. Indeed, for an arbitrary normal face $\partial\colon S\To T$ the induced map of abelian groups $\partial^*\colon(\Gamma C)_T\To (\Gamma C)_S$ vanishes on $C_T^{\id_T}$ since $F_C(\partial)=0$ by definition.
\end{proof}

\begin{prop}\label{unit}
Let $A$ be a dendroidal abelian group and $r\colon T\twoheadrightarrow R$ an epimorphism in $\rpd$. If we define $(\Psi_A)_T^r$ to be the composite
\[
\xymatrix{(NA)_R^r\ \ar@{>->}[r] &A_R\ar[r]^-{r^*} &A_T},
\]
then the induced map $(\Psi_A)_T\colon(\Gamma NA)_T\To A_T$ is an isomorphism which is natural in both $A$ and $T$.
\end{prop}
\begin{proof}
The proof of naturality is routinely verified. In what follows we will write $\Psi_T$ instead of $(\Psi_A)_T$ to simplify the notation. The first observation is that $\Psi_T$ decomposes as a direct sum
\[
 N\Psi_T\oplus D\Psi_T\colon (N\Gamma (NA))_T\oplus (D\Gamma (NA))_T\To (NA)_T\oplus (DA)_T,
\]
and by Proposition \ref{counit} the $N\Psi_T$ component is $\id\colon (NA)_T\To (NA)_T$. Hence in order to conclude that $\Psi_T$ is an isomorphism, it is enough to prove that $D\Psi_T$ is surjective and injective.

We proceed by induction on the number $n$ of vertices of $T$. If $n=0$ then $T=\eta$ and $(D\Gamma NA)_T=0=(DA)_T$.
Suppose that $D\Psi_S$ is surjective and injective for every tree $S$ with less than $n$ vertices and let $T$ be a tree with $n$ vertices.
Let $x$ be in the image of $\sigma^*\colon A_S\To A_T$  for some degeneracy $\sigma\colon T\To S$ and look at the commutative diagram
\[
\xymatrix{(\Gamma N A)_S\ar[r]^-{\Psi_S}\ar[d]_{\sigma^*} & A_S\ar[d]^{\sigma^*}\\
(\Gamma N A)_T\ar[r]^-{\Psi_T} &A_T.}
\]
Since $S$ has less vertices than $T$, we have that $D\Psi_S$ is surjective and $x=D\Psi_T(y)$ for some $y\in(D\Gamma N A)_T$. Hence $D\Psi_T$ is surjective.

Let us prove that $D\Psi_T$ is also injective. Suppose that $x\in \ker D\Psi_T$ and write
\[
x=\sum_{ \stackrel {r}{T\twoheadrightarrow R}} x^r.
\]
For any tree $R$ let $\X_{T,R}=\{r\colon T\twoheadrightarrow R\mid x^r\ne 0\}$ and suppose that $\X_{T,R}\neq~\emptyset$ for every $R$.
First we deal with the special case when $\X_{T,R}$ consists only of
epimorphisms $r$ occurring in the same maximal linear part of $T$.
For each $r\colon T\twoheadrightarrow R$ we choose a specific
section $d_r$, as follows. If $\sigma\colon T\to T\backslash v$ is a
degeneracy and the vertex $v$ has adjacent edges $e,f$ with $f$
situated above $e$
\[
\xy<0.08cm,0cm>:
(0,10.5)*{\vdots};
(0,-8)*{\vdots};
(0,6)*=0{}="1";
(0,0)*=0{\bullet}="2";
(0,-6)*=0{}="3";
(2,0)*{\scriptstyle v};
(-2,3)*{\scriptstyle f};
(-2,-3)*{\scriptstyle e};
"1";"2" **\dir{-};
"2";"3" **\dir{-};
\endxy
\]
then there is a unique face map $\partial\colon T\backslash v \To T$
which omits the edge $e$. This face map satisfies
$\sigma\partial=\id_{T\backslash v}$. If $r$ decomposes as
$r=\sigma_1\circ\sigma_2\circ\cdots\circ\sigma_k$ into degeneracies
then define the section
$d_r=\partial_k\circ\cdots\circ\partial_2\circ\partial_1$ of $r$
where $\partial_i$ is picked for $\sigma_i$ in the way described
above. Note that the dendroidal identities ensure that any other
decomposition of $r$ yields the same section.

Next we define a partial order on $\X_{T,R}$ as follows. If $r,s\in \X_{T,R}$ and for every edge $e$ sitting on the relevant linear part of $R$ the edge $d_r(e)$ is equal to or is below $d_s(e)$, then we say that $r\le s$. We observe that
\begin{equation}\label{order}
sd_r=\id_R \textnormal{ for some } r, s\in \X_{T,R} \textnormal{ implies }r\le s.
\end{equation}

Pick a maximal $r\in \X_{T,R}$ with respect to the order defined above. By (\ref{order}) and the maximality of $r$, the map $d_r$ satisfies
that $sd_r=\id$ implies $s=r$, hence we can conclude that in the commutative diagram
\[
\xymatrix{(\Gamma N A)_T\ar[r]^-{\Psi_T}\ar[d]_{d_r^*} & A_T\ar[d]^{d_r^*}\\
(\Gamma N A)_R\ar[r]^-{\Psi_R} & A_R,}
\]
on the left-hand side $N(d_r^*(x))=x^r$. Since $d_r^*\Psi_T=\id$ and $N\Psi_R=\id$, we infer that $x^r=0$, which is a contradiction.

We still need to deal with the general case, when $\X_{T,R}$ contains epimorphisms which are not necessarily situated on a fixed maximal linear part of $T$. To do so, suppose that $T$ has $k$ different maximal linear parts $L_{n_i}\rightarrowtail T$, $i\in\{1,\ldots,k\}$. Decompose each $r\in \X_{T,R}$ to $r=r_1\circ\cdots\circ r_k$ where $r_k\colon T\twoheadrightarrow R_{k-1}$ is located on the maximal linear part $L_{n_k}\rightarrowtail T$, $r_{k-1}\colon R_{k-1}\twoheadrightarrow R_{k-2}$ sits on the obvious maximal linear part $L_{n_{k-1}}\rightarrowtail R_{k-1}$ of the intermediate tree $R_{k-1}$, and so on. (Note that an epimorphism $r_i$ that appears in such a decomposition can be the identity.) We  can define a section $d_r$ of $r$ in the same way as in the special case above, moreover $d_r$ decomposes as
\[
d_r=d_{r_1}\circ\cdots \circ d_{r_k},
\]
where $d_{r_i}$ is a section of $r_i$. We can define a partial order on $\X_{T,R}$ as follows. Let $r$, $s\in \X_{T,R}$ have the associated decompositions
\[
r=r_1\circ\cdots \circ r_k \quad\textnormal{and}\quad s=s_1\circ\cdots \circ s_k.
\]
We say that $r\le s$ if for every $i$ there exist intermediate trees $R_i, R_{i-1}$ such that $s_i, r_i\in \X_{R_i,R_{i-1}}$ and $r_i\le s_i$ in the partial order defined in the special case above.

Again, if $sd_r=\id_R$ then $r\le s$ and we can mimic the rest of the proof of the special case to conclude that $x^r=0$ for a maximal $r\in \X_{T,R}$, arriving again to a contradiction.
\end{proof}

Now we are ready to prove the Dold-Kan correspondence theorem.

\begin{thm}\label{dold kan}
 The functors $N\colon \dAb\To \dCh$ and $\Gamma\colon\dCh\To \dAb$ form an equivalence of categories.
\end{thm}
\begin{proof}
Propositions \ref{counit} and \ref{unit} imply that $N$ and $\Gamma$ together form an adjoint equivalence where the unit of the adjunction is the natural isomorphism $\Psi^{-1}$ of Proposition \ref{unit} and the counit is the
identity.
\end{proof}

\end{document}